\documentclass[11pt, reqno]{amsart}
\usepackage{eurosym}
\usepackage{graphicx,amsmath,amssymb,epsfig, float}
\usepackage{amsfonts, lscape}
\usepackage{color}
\usepackage{bbm,threeparttable,subcaption}
\usepackage{array}
\usepackage{verbatim}

\theoremstyle{plain}

\newtheorem{corollary}{Corollary}
\newtheorem*{defUN}{Definition}

\newtheorem*{thmUN}{Theorem}

\newtheorem{proposition}{Proposition}

\theoremstyle{definition}

\theoremstyle{definition}
\newtheorem{remark}{Remark}
\graphicspath{ {FigsEtc/} }
\newcolumntype{M}[1]{>{\centering\arraybackslash}m{#1}}
\input{tcilatex}

\begin{document}
\def\sym#1{\ifmmode^{#1}\else\(^{#1}\)\fi}

\title{On the representation of the nested logit model}
\author{Alfred Galichon{$^{\S }$}}
\date{October 14, 2020. (First draft: 7/2019). Funding from NSF grant DMS-1716489 and
ERC-CoG grant No. 866274 is
acknowledged. The author is thankful to the editor (Peter Phillips) and two
anonymous reviewers, as well as Kenneth Train, Mogens Fosgerau, Lars-G\"{o}%
ran Mattsson, Bernard Salani\'{e} and J\"{o}rgen Weibull for helpful
comments, and especially to Matt Shum for pointing out the important
reference to Cardell's paper after a first version of this paper was
circulated.\\
{\indent$^{\S }$New York University, departments of Economics and of
Mathematics; ag133@nyu.edu}}

\begin{abstract}
We give a two-line proof of a long-standing conjecture of Ben-Akiva and
Lerman (1985) regarding the random utility representation of the nested
logit model, thus providing a renewed and straightforward textbook treatment
of that model. As an application, we provide a closed-form formula for the
correlation between two Fr\'{e}chet random variables coupled by a Gumbel
copula.

\vspace{0.1cm} \emph{Keywords:} random utility models, multinomial choice,
nested logit model

\vspace{0.1cm} \emph{JEL Classification:} C51, C60
\end{abstract}

\maketitle

\section{Positive stable distributions and Gumbel distributions\label%
{sec:mainthm}}

Ben-Akiva (1973) introduced the nested logit model in his PhD thesis. In the
light of McFadden's (1978) theory of the generalized extreme value
distribution, one may view the nested logit model as a choice model over
various alternatives $j\in \bigcup_{n\in \mathcal{N}}\mathcal{N}_{n}$, where
$\mathcal{N}_{n}$ is a \emph{nest}, a subset of alternatives. The utility
associated with alternative $j$ in nest $n$ will be $U_{nj}+\varepsilon
_{nj} $, where $U_{nj}$ is the deterministic (\textquotedblleft
systematic\textquotedblright ) part of the utility and $\varepsilon _{nj}$
is a random utility term. The nested logit model specifies that the random
vector $\left( \varepsilon _{nj}\right) $ has cumulative distribution
function (c.d.f.) $F_{\varepsilon }$ given by%
\begin{equation}
F_{\varepsilon }\left( \left( a_{j}\right) \right) =\exp \left( -\sum_{n\in
\mathcal{N}}\left( \sum_{j\in \mathcal{N}_{n}}e^{-\frac{a_{j}}{\lambda _{n}}%
}\right) ^{\lambda _{n}}\right) ,  \label{nestedLogit}
\end{equation}%
where $\lambda _{n}\in (0,1]$ is a measure of how correlated the random
utility terms are within nest $n$, $\lambda _{n}=1$ meaning full
independence, and the limit $\lambda _{n}\rightarrow 0$ meaning perfect
correlation within the nest. As pointed out in McFadden (1978), the
correlation between random utility terms within nest $n$ is \emph{not} given
by $1-\lambda _{n}$. The correct formula for the correlation ($1-\lambda
_{n}^{2}$) had been obtained by Tiago de Oliveira (1958) from a non-trivial
analytic derivation.

In their 1985 book, Ben-Akiva and Lerman proposed (section 10.3) a heuristic
interpretation of the nested logit model. Calling $\mathcal{G}\left( \mu
,\beta \right) $ the distribution of the Gumbel random variable with
location-scale parameters $\left( \mu ,\beta \right) $, whose cumulative
distribution function (c.d.f.) is thus $F_{\epsilon }\left( z\right) =\exp
\left( -\exp \left( -\left( z-\mu \right) /\beta \right) \right) $,
Ben-Akiva and Lerman (p. 282)\footnote{%
Cardell (1997) mentions that a similar conjecture appeared in a 1975
manuscript of his, but we haven't been able to access that manuscript.}
conjectured that if $\epsilon $ has distribution $\mathcal{G}\left(
0,1\right) $, there is a random variable $\eta $ independent from $\epsilon $%
\ such that $\lambda \epsilon +\eta $ has distribution $\mathcal{G}\left(
0,1\right) $. This conjecture allowed them to derive the correlation
structure between the random utilities within a nest. The existence of $\eta
$ was the object of work by Cardell (1997) who provides a formula for the
density of $\eta $ as the sum of an alternating series by the inversion of
its characteristics function. However, Cardell's proof of existence is
incomplete, as it omits to show that the candidate density obtained as the
sum of this series is indeed nonnegative. In the present paper we provide a
complete, very short proof involving \textquotedblleft positive stable
distributions\textquotedblright , which are defined as follows in Feller
(1971), sections VI.1 and XIII.6\footnote{%
The existence of positive stable distributions was established in Pollard
(1948), using a lemma in Bochner (1937).}:

\begin{defUN}
For $0<\lambda \leq 1$, there is a probability distribution $\mathcal{P}%
\left( \lambda \right) $ over $\mathbb{R}_{+}$, called \emph{positive stable
distribution with parameter $\lambda $}, such that if $\left(
Z_{1},...Z_{N}\right) $ are i.i.d. draws from that distribution, then for
positive reals $\alpha _{1},...,\alpha _{N}$,
\begin{equation}
\frac{\alpha _{1}Z_{1}+...+\alpha _{N}Z_{N}}{\left( \alpha _{1}^{\lambda
}+...+\alpha _{N}^{\lambda }\right) ^{1/\lambda }}\sim \mathcal{P}\left(
\lambda \right) ,  \label{plus-stability}
\end{equation}%
and the Laplace distribution of $\mathcal{P}\left( \lambda \right) $ is $%
\mathbb{E}_{\mathcal{P}\left( \lambda \right) }\left[ \exp \left( -tZ\right) %
\right] =\exp \left( -t^{\lambda }\right) $.
\end{defUN}

Here and throughout the paper we use $\sim $ to indicate \textquotedblleft
distributed as\textquotedblright . Our main result connects the Gumbel
distribution and positive stable distributions:

\begin{thmUN}
For $0<\lambda \leq 1$, if random variables $\epsilon $ and $Z$ are
independent, and $\epsilon \sim \mathcal{G}\left( 0,1\right) $, and $Z\sim
\mathcal{P}\left( \lambda \right) $ has a positive stable distribution with
parameter $\lambda $, then $\epsilon +\log Z\sim \mathcal{G}\left(
0,1/\lambda \right) $.
\end{thmUN}

\begin{proof}
Assuming $\epsilon \sim \mathcal{G}\left( 0,1\right) $, and $Z\sim \mathcal{P%
}\left( \lambda \right) $, the c.d.f. of $\epsilon +\log Z$ is expressed as%
\begin{eqnarray*}
&&\Pr \left( \epsilon +\log Z\leq x\right) =\mathbb{E}\left[ \mathbb{E}\left[
1\left\{ \epsilon \leq x-\log Z\right\} |Z\right] \right] =\mathbb{E}\left[
F_{\epsilon }\left( x-\log Z\right) \right] \\
&=&\mathbb{E}\left[ \exp \left( -\exp \left( \log Z-x\right) \right) \right]
=\mathbb{E}\left[ \exp \left( -Z\exp \left( -x\right) \right) \right] =\exp
\left( -\exp \left( -\lambda x\right) \right) ,
\end{eqnarray*}%
which is the c.d.f. of the $\mathcal{G}\left( 0,1/\lambda \right) $
distribution.
\end{proof}

As a result, we can deduce the moments of $\eta =\lambda \log Z$, as is done
in Ben-Akiva and Lerman (1985), section 10.3. The expectation of $\eta $ is $%
\mathbb{E}\left[ \eta \right] =\left( 1-\lambda \right) \mathbb{E}\left[
\epsilon \right] =\left( 1-\lambda \right) \gamma $ where $\gamma $ is
Euler's constant. The variance of $\eta $ is $var\left( \eta \right) =\left(
1-\lambda ^{2}\right) var\left( \epsilon \right) =\left( 1-\lambda
^{2}\right) \pi ^{2}/6$. More generally, the moment generating function of $%
\eta $ is given by $\mathbb{E}\left[ e^{t\eta }\right] =\Gamma \left(
1-t\right) /\Gamma \left( 1-\lambda t\right) $, that is, for $\kappa \in
\left( 0,\lambda \right) $,%
\begin{equation}
\mathbb{E}\left[ Z^{\kappa }\right] =\frac{\Gamma \left( 1-\kappa /\lambda
\right) }{\Gamma \left( 1-\kappa \right) }.  \label{momentsOfZ}
\end{equation}

The probability distribution function $f_{\lambda }$ associated with $%
\mathcal{P}\left( \lambda \right) $ is in general not available in closed
form, except for a few cases, namely $\lambda =1/3,1/2$ and $2/3$, in which
case one obtains
\begin{eqnarray*}
f_{1/3}\left( x\right) &=&\frac{1}{3^{1/3}x^{4/3}}Ai\left( \frac{1}{%
3^{1/3}x^{1/3}}\right) \\
f_{1/2}\left( x\right) &=&\frac{1}{2\pi ^{1/2}}\frac{1}{x^{3/2}}\exp \left(
\frac{-1}{4x}\right) \\
f_{2/3}\left( x\right) &=&\frac{2^{4/3}\exp \left( \frac{-4}{27x^{2}}\right)
}{3^{3/2}\pi ^{1/2}x^{7/3}}U\left( \frac{1}{2},\frac{1}{6},\frac{4}{27x^{2}}%
\right)
\end{eqnarray*}%
where $Ai$ is the Airy function and $U$ is the confluent hypergeometric
function; see Dishon and Bendler (1990). In general, $f_{\lambda }$ cannot
be given an explicit expression, but can be expressed as a convergent series
expansion using Humbert's (1945) formula
\begin{equation*}
f_{\lambda }\left( x\right) =\frac{-1}{\pi }\sum_{k=0}^{\infty }\frac{\left(
-1\right) ^{k}}{k!}\sin \left( k\pi \lambda \right) \frac{\Gamma \left(
\lambda k+1\right) }{x^{\lambda k+1}},
\end{equation*}%
where $\Gamma $ is the Gamma function, classically defined as%
\begin{equation*}
\Gamma \left( z\right) =\int_{0}^{+\infty }t^{z-1}e^{-t}dt.
\end{equation*}

\begin{remark}[Relating plus-stability and max-stability]
Setting $\eta _{\lambda }=\lambda \log Z_{\lambda }$ for $Z_{\lambda }\sim
\mathcal{P}\left( \lambda \right) $, and taking $\left( \tilde{\eta}%
_{\lambda ,j}\right) _{j\in \mathcal{J}}$ a vector of independent copies of $%
\eta _{\lambda }$, it follows from~(\ref{plus-stability})\ that we have%
\begin{equation*}
\lambda \log \left( \sum_{j\in \mathcal{J}}\exp \left( \frac{u_{j}+\eta
_{\lambda ,j}}{\lambda }\right) \right) =_{D}\log \left( \sum_{j\in \mathcal{%
J}}e^{u_{j}}\right) +\eta _{\lambda },
\end{equation*}%
and it follows from $\eta _{\lambda }+\lambda \epsilon \sim \mathcal{G}%
\left( 0,1\right) $ that $\eta _{\lambda }$ converges weakly to $\mathcal{G}%
\left( 0,1\right) $. As the log-sum-exp (a.k.a. smooth-max) converges to the
max when $\lambda \rightarrow 0$, we recover the max-stability of the Gumbel
distribution as a consequence.
\end{remark}

\bigskip

As noted by Cardell (1997, lemma 1), another interesting consequence of the
theorem is as follows.

\begin{remark}
Take $\lambda _{1},\lambda _{2}\in \left( 0,1\right) $. From the theorem, it
holds that if $Z_{1}\sim \mathcal{P}\left( \lambda _{1}\right) $ and $%
\epsilon _{1}\sim \mathcal{G}\left( 0,1\right) $ are independent, then $%
\epsilon _{2}=\lambda _{1}\left( \epsilon _{1}+\log Z_{1}\right) \sim
\mathcal{G}\left( 0,1\right) $. Thus, if $Z_{2}\sim \mathcal{P}\left(
\lambda _{2}\right) $ is independent from $\epsilon _{1}$ and $Z_{1}$, then $%
\lambda _{2}\left( \epsilon _{2}+\log Z_{2}\right) \sim \mathcal{G}\left(
0,1\right) $, then
\begin{equation*}
\lambda _{2}\lambda _{1}\epsilon _{1}+\lambda _{2}\lambda _{1}\log
Z_{1}+\lambda _{2}\log Z_{2}\sim \mathcal{G}\left( 0,1\right) .
\end{equation*}%
As a result, $\lambda _{2}\lambda _{1}\log Z_{1}+\lambda _{2}\log Z_{2}$ has
the same distribution as $\lambda _{1}\lambda _{2}\log Z_{12}$, where $%
Z_{12}\sim \mathcal{P}\left( \lambda _{1}\lambda _{2}\right) $. Thus for $%
\lambda _{1},\lambda _{2}\in \left( 0,1\right) $, if $Z_{1}\sim \mathcal{P}%
\left( \lambda _{1}\right) $ and $Z_{2}\sim \mathcal{P}\left( \lambda
_{2}\right) $ are independent, then
\begin{equation*}
Z_{1}Z_{2}^{1/\lambda _{1}}\sim \mathcal{P}\left( \lambda _{1}\lambda
_{2}\right) .
\end{equation*}
\end{remark}

\section{Representation of the nested logit model}

We can use the result of the previous section to provide rigorous ground for
the representation of the random utility associated with the nested logit
model, which we carry first in the context of a single layer of nests,
before moving on to the case of multiple layers.

\subsection{Single layer of nests}

The model described in section~\ref{sec:mainthm} is a single nest model,
where, as it is well-known, if $\left( \varepsilon _{nj}\right) \sim
F_{\varepsilon }$ has the distribution given by~(\ref{nestedLogit}), then
the probability that $j\in \mathcal{N}_{n}$ maximizes $U_{j^{\prime
}}+\varepsilon _{j^{\prime }}$ over $j^{\prime }\in \mathcal{N}_{n^{\prime
}} $ and $n^{\prime }\in \mathcal{N}$, is given by%
\begin{equation}
\pi _{j}=\frac{\left( \sum_{j^{\prime }\in \mathcal{N}_{n}}e^{\frac{%
U_{j^{\prime }}}{\lambda _{n}}}\right) ^{\lambda _{n}}}{\sum_{\substack{ %
n^{\prime }\in \mathcal{N}  \\ j^{\prime }\in \mathcal{N}_{n^{\prime }}}}%
\left( \sum_{j^{\prime }\in \mathcal{N}_{n^{\prime }}}e^{\frac{U_{j^{\prime
}}}{\lambda _{n^{\prime }}}}\right) ^{\lambda _{n^{\prime }}}}\frac{e^{\frac{%
U_{j}}{\lambda _{n}}}}{\sum_{j^{\prime }\in \mathcal{N}_{n}}e^{\frac{%
U_{j^{\prime }}}{\lambda _{n}}}},  \label{mktSharesSingleLayerNest}
\end{equation}%
which is the product of the probability that nest $n$ is chosen by the
conditional probability that $j$ is chosen within nest $n$. It is also
well-known that in this case, the Emax operator $G\left( U\right) =\mathbb{E}%
\left[ \max_{n\in \mathcal{N},j\in \mathcal{N}_{n}}\left\{ U_{j}+\varepsilon
_{j}\right\} \right] $ has expression%
\begin{equation}
G\left( U\right) =\log \sum_{n\in \mathcal{N}}\left( \sum_{j\in \mathcal{N}%
_{n}}e^{\frac{U_{j}}{\lambda _{n}}}\right) ^{\lambda _{n}},
\label{EmaxSingleLayerNest}
\end{equation}%
and we can recover that, as predicted by the Daly-Zachary-Williams theorem,
the expression of $\pi _{j}$ (\ref{mktSharesSingleLayerNest}) is indeed
obtained as the derivative of $G$ with respect to $U_{j}$.

\bigskip

With these reminders in mind, we can use the results in section~\ref%
{sec:mainthm} to provide a factor representation of the random vector $%
\left( \varepsilon _{j}\right) $.

\begin{proposition}
\label{prop:representation}If $\left( \varepsilon _{j}\right) $ is a random
vector distributed according to the nested logit distribution~(\ref%
{nestedLogit}), then one can represent the random vector $\left( \varepsilon
_{j}\right) $ as
\begin{equation}
\varepsilon _{j}=\lambda _{n}\left( \epsilon _{j}+\log Z_{n}\right)
\label{repNested}
\end{equation}%
where $\epsilon _{j}\sim \mathcal{G}\left( 0,1\right) $ for $j\in \mathcal{N}%
_{n},n\in \mathcal{N}$, and $Z_{n}\sim \mathcal{P}\left( \lambda _{n}\right)
$ for $n\in \mathcal{N}$ are all independently distributed. In particular,
the correlation between $\varepsilon _{j}$ and $\varepsilon _{j^{\prime }}$
for $j\in \mathcal{N}_{n}$ and $j^{\prime }\in \mathcal{N}_{n^{\prime }}$ is
$\left( 1-\lambda _{n}^{2}\right) $ if $n=n^{\prime }$, $0$ otherwise.
\end{proposition}

Note that the correlation structure was first rigorously derived by Tiago de
Oliveira (1958) using very different techniques, see also p. 266 of his 1997
monograph.

\begin{proof}
Letting $\eta _{n}=\lambda _{n}\exp Z_{n}$, the c.d.f. of the vector $\left(
\lambda _{n}\epsilon _{j}+\eta _{n}\right) _{j\in \mathcal{N}_{n},n\in
\mathcal{N}}$ is
\begin{equation*}
\Pr \left( \lambda _{n}\epsilon _{j}+\eta _{n}\leq a_{j}~\forall n\in
\mathcal{N},\forall j\in \mathcal{N}_{n}\right) =\Pr \left( \eta _{n}\leq
-\max_{j\in \mathcal{N}_{n}}\left\{ -a_{j}+\lambda _{n}\epsilon _{j}\right\}
\forall n\in \mathcal{N}\right) .
\end{equation*}%
But $\max_{j\in \mathcal{N}_{n}}\left\{ -a_{nj}+\lambda _{n}\epsilon
_{j}\right\} $ has the same distribution as $\lambda _{n}\log \sum_{j\in
\mathcal{N}_{n}}\exp \left( \frac{-a_{j}}{\lambda _{n}}\right) +\lambda _{n}%
\hat{\epsilon}_{n}$ where $\hat{\epsilon}_{n}$ has a $\mathcal{G}\left(
0,1\right) $ distribution and the $\hat{\epsilon}_{n}$ and the $\eta _{n}$
terms are all independent. Hence the displayed equation equals
\begin{equation*}
\Pr \left( \lambda _{n}\epsilon _{n}+\eta _{n}\leq -\lambda _{n}\log
\sum_{j\in \mathcal{N}_{n}}\exp \left( \frac{-a_{j}}{\lambda _{n}}\right)
\forall n\in \mathcal{N}\right)
\end{equation*}%
which, as $\left( \lambda _{n}\epsilon _{n}+\eta _{n}\right) \ $are i.i.d.
with distribution $\mathcal{G}\left( 0,1\right) $, is the same as
expression~(\ref{nestedLogit}).
\end{proof}

\begin{remark}
Representation~(\ref{repNested}) allows to view the nested logit model as a
mixed logit model. Indeed, it implies%
\begin{eqnarray*}
&&\Pr \left( U_{j}+\lambda _{n}\left( \epsilon _{j}+\log Z_{n}\right) \geq
U_{j^{\prime }}+\lambda _{n^{\prime }}\left( \epsilon _{j^{\prime }}+\log
Z_{n^{\prime }}\right) \forall n^{\prime }\in \mathcal{N},j^{\prime }\in
\mathcal{N}_{n^{\prime }}\right) \\
&=&\mathbb{E}\left[ \frac{Z_{n}\exp \left( U_{j}/\lambda _{n}\right) }{%
\sum_{n^{\prime }\in \mathcal{N},j^{\prime }\in \mathcal{N}_{n^{\prime
}}}Z_{n^{\prime }}\exp \left( U_{j^{\prime }}/\lambda _{n^{\prime }}\right) }%
\right] ,
\end{eqnarray*}%
where $Z_{n}\sim \mathcal{P}\left( \lambda _{n}\right) $, and therefore, if
the $Z_{n}^{k}$, $k\in \left\{ 1,...,K\right\} $ are $K$ draws from the $%
\mathcal{P}\left( \lambda _{n}\right) $ distribution, one can approximate
the previous expression by%
\begin{equation*}
\sum_{k=1}^{K}\frac{Z_{n}^{k}\exp \left( U_{j}/\lambda _{n}\right) }{%
\sum_{k^{\prime }=1}^{K}\sum_{n^{\prime }\in \mathcal{N},j^{\prime }\in
\mathcal{N}_{n^{\prime }}}Z_{n^{\prime }}^{k^{\prime }}\exp \left(
U_{j^{\prime }}/\lambda _{n^{\prime }}\right) }.
\end{equation*}%
The simulation of positive stable distribution is discussed in Ridout
(2009), where an algorithm is proposed and implemented in R.
\end{remark}

\subsection{Multiple layers of nests\label{par:multiple-nests}}

The nested logit model is defined in its full generality on an arborescence
of nests. In order to facilitate our description of the model, we need to
recall the standard mathematical terminology of trees.

\subsubsection{Trees}

Consider a finite set of \emph{nodes} $\mathcal{Z}$ with more than one
element, and call \emph{root} one of its nodes, denoted $0$. An \emph{%
arborescence} (a.k.a. \textquotedblleft directed rooted
tree\textquotedblright ) on $\mathcal{Z}$ is a directed graph such that
there is exactly one directed path from $0$ to any node $z$. Nodes that have
no outward arcs are called \emph{leaves} -- which we will identify in our
case as \textquotedblleft alternatives.\textquotedblright\ The set of
leaves/alternatives is denoted $\mathcal{J}$. We shall call \emph{nests} the
nodes which are not alternatives, and denote by $\mathcal{N}$ the set of
nests. We have therefore $\mathcal{Z}=\mathcal{N}\cup \mathcal{J}$, and $%
0\in \mathcal{N}$.

Each nest $n$ has a set $\mathcal{N}_{n}$ set of child nests and $\mathcal{J}%
_{n}$ set of child alternatives. Both $\mathcal{N}_{n}$ and $\mathcal{J}_{n}$
can be empty, but $\mathcal{N}_{n}\cup \mathcal{J}_{n}$ cannot. Alternatives
have no children: they are the \textquotedblleft leaves\textquotedblright\
of the arborescence.

The parent node of any node $z$ which is not the root is a nest denoted $%
n_{z}$.

The \emph{depth} $d_{z}$ of a node $z$ is the number of edges from the nest
to the root nest $0$. The depth of the root nest is zero.

The \emph{height} $h_{z}$ of a node $z$ is the number of edges on the
longest path from the nest to a leaf (alternative). In particular,
alternatives have height zero, and nodes that have only alternatives (and no
nests) as children, have height one. The height of the arborescence is the
height of $0$.

If a nest $n$ has depth $d$, the path from the root nest to $n$ along the
arborescence is denoted $n\left( 1\right) ,...,n\left( d\right) $, where $%
n\left( 1\right) \in \mathcal{N}_{0}$ and $n\left( k+1\right) \in \mathcal{N}%
_{n\left( k\right) }$, and $n\left( d\right) =n$. Note that by convention,
the root nest $0$ is omitted from this path.

If there is a path from $z$ to $z^{\prime }$ along the arborescence, one
says that $z$ is an \emph{ancestor} of $z^{\prime }$ and $z^{\prime }$ is a
\emph{descendant} of $z$.

We denote by $\widetilde{\mathcal{J}}_{z}$ the set of alternatives that are
descendant of node $z$. This is to be contrasted by $\mathcal{J}_{z}$, which
is the set of alternatives that are children of $z$. Of course, $\widetilde{%
\mathcal{J}}_{0}=\mathcal{J}$. Note that $\widetilde{\mathcal{J}}%
_{j}=\left\{ j\right\} $ for $j\in \mathcal{J}$.

The \emph{lowest common ancestor} to two nodes is the node which is an
ancestor of both of them and which has the largest depth.

\subsubsection{Reminders on the nested logit model}

For any nest $n\in \mathcal{N}$, we define $\lambda _{n}\in (0,1]\,$. We
assume that $\lambda _{0}=1$, and we define
\begin{equation*}
\Lambda _{n}=\prod_{t=1}^{d_{n}}\lambda _{n\left( t\right) }.
\end{equation*}

The nested logit model takes as an input the \emph{systematic utility} $%
U_{j} $ associated with each alternative $j\in \mathcal{J}$, and returns the
vector of \emph{choice probability} $\pi _{j}$ associated with each $j$, a
probability vector on $\mathcal{J}$. As is well-known, $\pi $ is computed as
follows:

1. Computing utilities by \emph{backward induction}. One defines the payoff
vector $u$ on all the nodes $z\in \mathcal{Z}=\mathcal{N}\cup \mathcal{J}$
recursively in ascending height order by%
\begin{equation}
\left\{
\begin{array}{l}
u_{z}=U_{z}\text{~for }z\in \mathcal{J}\text{ (i.e. }h_{z}=0\text{)} \\
u_{n}=\Lambda _{n}\log \left( \sum_{z\in \mathcal{N}_{n}\cup \mathcal{J}%
_{n}}\exp \left( u_{z}/\Lambda _{n}\right) \right) ~\text{for }h_{n}>1\text{.%
}%
\end{array}%
\right.  \label{utils}
\end{equation}

This is well-defined because the children of a nest of height $t$ has height
at most $t-1$.

2. Computing choice probabilities by \emph{forward induction}. Define $\pi
_{z}$ in ascending depth order by%
\begin{equation}
\left\{
\begin{array}{l}
\pi _{0}=1 \\
\pi _{z}=\pi _{n_{z}}\frac{\exp \left( u_{z}/\Lambda _{n_{z}}\right) }{%
\sum_{z^{\prime }\in \mathcal{N}_{n_{z}}\cup \mathcal{J}_{n_{z}}}\exp \left(
u_{z^{\prime }}/\Lambda _{n_{z}}\right) }\text{ for }z\neq 0.%
\end{array}%
\right.  \label{ccps}
\end{equation}

This is well-defined because the parent of a node of depth $t$ has depth $%
t-1 $. In particular, this procedure defines the probabilities $\pi _{j}$ on
the alternatives $\mathcal{J}$, which are the leaves of the arborescence.
Note that combining~(\ref{utils}) and~(\ref{ccps}) yields the \emph{log-odds
ratio formula}%
\begin{equation}
\frac{\pi _{z}}{\pi _{n}}=\exp \left( \frac{u_{z}-u_{n}}{\Lambda _{n}}%
\right) ,~\text{for }z\in \mathcal{N}_{n}\cup \mathcal{J}_{n}.  \label{gibbs}
\end{equation}

\bigskip

It is well-known that the nested logit model is an additive random utility
model. Consider a random vector $\varepsilon $ whose c.d.f. is defined by
backward induction by%
\begin{eqnarray}
&&\Pr \left( \varepsilon _{j}\leq A_{j}~\forall j\in \mathcal{J}\right)
=\exp \left( -\exp \left( -a_{0}\right) \right) ,
\label{nestedLogitMultiLayer} \\
&&\text{where }a\in \mathbb{R}^{\mathcal{Z}}\text{ is defined recursively by}
\notag \\
&&\left\{
\begin{array}{l}
a_{z}=A_{z}\text{~for }z\in \mathcal{J}\text{ (i.e. }h_{z}=0\text{)} \\
a_{n}=-\Lambda _{n}\log \left( \sum_{z\in \mathcal{N}_{n}\cup \mathcal{J}%
_{n}}\exp \left( -a_{z}/\Lambda _{n}\right) \right) ~\text{for }h_{n}>1\text{%
.}%
\end{array}%
\right.  \notag
\end{eqnarray}

Then the choice probabilities $\left( \pi _{j}\right) $ determined by the
backward-forward procedure coincide with the probabilities induced by the
additive random utility model associated with $\varepsilon $, that is%
\begin{equation*}
\pi _{j}=\Pr \left( U_{j}+\varepsilon _{j}\geq U_{j^{\prime }}+\varepsilon
_{j^{\prime }}~\forall j^{\prime }\in \mathcal{J}\right) .
\end{equation*}

Another well-known relation between the distribution of $\varepsilon $ and
the payoff vector $u$ is given by the Emax operator defined (see McFadden,
1978) by%
\begin{equation*}
G\left( U\right) :=G_{0}\left( U\right) :=\mathbb{E}\left[ \max_{j\in
\mathcal{J}}\left\{ U_{j}+\varepsilon _{j}\right\} \right] ,
\end{equation*}%
and more generally, the restricted Emax at $n$ is defined by taking the
maximum over alternatives that are descendent of $n$, that is%
\begin{equation*}
G_{n}\left( U\right) :=\mathbb{E}\left[ \max_{j\in \widetilde{\mathcal{J}}%
_{n}}\left\{ U_{j}+\varepsilon _{j}\right\} \right]
\end{equation*}%
and one has
\begin{equation*}
u_{n}=G_{n}\left( U\right)
\end{equation*}%
where $u_{n}$ is computed as in~(\ref{utils}). One has%
\begin{equation*}
\frac{\partial G_{n}\left( U\right) }{\partial U_{j}}=\frac{\pi _{j}}{\pi
_{n}}1\left\{ j\in \widetilde{\mathcal{J}}_{n}\right\} ,
\end{equation*}%
and in particular setting $n=0$ one recovers $\partial G\left( U\right)
/\partial U_{j}=\pi _{j}$, which is the Daly-Zachary-Williams theorem.

\subsubsection{The representation result}

From the main theorem in section~\ref{sec:mainthm}, we are able to derive
the following representation.

\begin{proposition}
\label{prop:representation-multiLayer}If $\left( \varepsilon _{j}\right) $
is a random vector distributed according to the nested logit distribution~(%
\ref{nestedLogitMultiLayer}), then one can represent the random vector $%
\left( \varepsilon _{j}\right) _{j\in \mathcal{J}}$ as
\begin{equation}
\varepsilon _{j}=\sum_{t=1}^{d_{n_{j}}}\Lambda _{n_{j}\left( t\right) }\log
Z_{n_{j}\left( t\right) }+\Lambda _{n_{j}}\epsilon _{j}
\label{representation-multiLayer}
\end{equation}%
where $\epsilon _{j}\sim \mathcal{G}\left( 0,1\right) $ and $Z_{n\left(
t\right) }\sim \mathcal{P}\left( \lambda _{t}\right) $ are all independent.
\end{proposition}

\begin{proof}
By induction on $h\geq 0$ and less than the height of the arborescence, one
shows that the probability that
\begin{equation*}
\sum_{t=1}^{d_{n}}\Lambda _{n\left( t\right) }\log Z_{n\left( t\right)
}+\Lambda _{n}\epsilon _{j}\leq a_{j}\forall n\in \mathcal{N},\forall j\in
\mathcal{J}_{n}
\end{equation*}%
coincides with the probability that
\begin{equation}
\begin{array}{l}
\sum_{t=1}^{d_{n}}\Lambda _{n\left( t\right) }\log Z_{n\left( t\right)
}+\Lambda _{n}\epsilon _{j}\leq a_{j},\forall n\in \mathcal{N}%
:h_{n}>h,\forall j\in \mathcal{J}_{n}\text{ and} \\
\sum_{t=1}^{d_{n}}\Lambda _{n\left( t\right) }\log Z_{n\left( t\right)
}+\Lambda _{n}\epsilon _{n^{\prime }}\leq A_{n^{\prime }},\forall n\in
\mathcal{N}:h_{n}=h+1,\forall n^{\prime }\in \mathcal{N}_{n}%
\end{array}
\label{event-1}
\end{equation}%
where $\epsilon _{z}$ is defined on all the nodes $z\in \mathcal{Z}=\mathcal{%
N}\cup \mathcal{J}$ recursively in increasing height order by
\begin{equation*}
\left\{
\begin{array}{l}
\epsilon _{z}=\epsilon _{z}\text{~for }z\in \mathcal{J}\text{ (i.e. }h_{z}=0%
\text{)} \\
\epsilon _{n}:=\lambda _{n}\log Z_{n}+\lambda _{n}\max_{z\in \mathcal{N}%
_{n}\cup \mathcal{J}_{n}}\left\{ -a_{z}/\Lambda _{n}+\epsilon _{z}\right\}
-\lambda _{n}\sum_{z\in \mathcal{N}_{n}\cup \mathcal{J}_{n}}\exp \left(
-a_{z}/\Lambda _{n}\right) ~\text{ for }h_{n}>1\text{.}%
\end{array}%
\right.
\end{equation*}

Base case: if $h_{n}=1$, then $\mathcal{N}_{n}=\emptyset $, and the base
case ($h=0$) is proven.

Inductive step: assume the property proven up to rank $h$, and consider a
nest $n$ of height $h_{n}=h+1$. Then the nodes that are children of $n$ are
either alternatives, or other nests that have height at most $h$. Collecting
all the terms pertaining to $n$, these are
\begin{eqnarray*}
&&\sum_{t=1}^{d_{n}}\Lambda _{n\left( t\right) }\log Z_{n\left( t\right)
}+\Lambda _{n}\epsilon _{j}\leq A_{j},\forall j\in \mathcal{J}_{n}\text{, and%
} \\
&&\sum_{t=1}^{d_{n}}\Lambda _{n\left( t\right) }\log Z_{n\left( t\right)
}+\Lambda _{n}\epsilon _{n^{\prime }}\leq a_{n^{\prime }},\forall n^{\prime
}\in \mathcal{N}_{n}.
\end{eqnarray*}

The conditional probability that this occurs is the probability that
\begin{equation*}
\sum_{t=1}^{d_{n}}\Lambda _{n\left( t\right) }\log Z_{n\left( t\right)
}+\Lambda _{n}\tilde{\epsilon}_{n}\leq a_{n},
\end{equation*}

where we have set $\tilde{\epsilon}_{n}:=\max_{z\in \mathcal{N}_{n}\cup
\mathcal{J}_{n}}\left\{ -a_{z}/\Lambda _{n}+\epsilon _{z}\right\}
-\sum_{z\in \mathcal{N}_{n}\cup \mathcal{J}_{n}}\exp \left( -a_{z}/\Lambda
_{n}\right) $. This is the same as the probability that
\begin{equation*}
\sum_{t=1}^{d_{n\left( t-1\right) }}\Lambda _{n\left( t\right) }\log
Z_{n\left( t\right) }+\Lambda _{n}\log Z_{n}+\Lambda _{n}\tilde{\epsilon}%
_{n}\leq a_{n},
\end{equation*}

and setting $\epsilon _{n}:=\lambda _{n}\log Z_{n}+\lambda _{n}\tilde{%
\epsilon}_{n}^{1}$ yields%
\begin{equation*}
\sum_{t=1}^{d_{n\left( t-1\right) }}\Lambda _{n\left( t\right) }\log
Z_{n\left( t\right) }+\Lambda _{n\left( t-1\right) }\epsilon _{n}\leq a_{n}.
\end{equation*}

Collecting all $n$ with height $h_{n}=h+1$ allows to complete the proof of
the inductive step.
\end{proof}

\begin{corollary}
Under the assumptions of Proposition~\ref{prop:representation-multiLayer},
for $j\neq j^{\prime }$, the correlation between $\varepsilon _{j}$ and $%
\varepsilon _{j^{\prime }}$ is $1-\Lambda _{n}^{2}$ where $n$ is the lowest
common ancestor between $j$ and $j$'.
\end{corollary}

\begin{proof}
Let $n$ be the lowest common ancestor between $j$ and $j^{\prime }$, we get
that if $n=n_{j}\left( s\right) $
\begin{equation*}
\sum_{t=s}^{d_{n_{j}}}\frac{\Lambda _{n_{j}\left( t\right) }}{\Lambda _{n}}%
\log Z_{n_{j}\left( t\right) }+\frac{\Lambda _{n_{j}}}{\Lambda _{n}}\epsilon
_{j}
\end{equation*}%
is distributed as a $\mathcal{G}\left( 0,1\right) $ distribution, which one
shall denote $\epsilon _{n,j}$. The same applies to the representation of $%
\varepsilon _{j^{\prime }}$, and hence $(\varepsilon _{j},\varepsilon
_{j^{\prime }})$ can be represented as%
\begin{equation*}
\left\{
\begin{array}{c}
\varepsilon _{j}=\sum_{t=1}^{d_{n}}\Lambda _{n\left( t\right) }\log
Z_{n\left( t\right) }+\Lambda _{n}\epsilon _{n,j} \\
\varepsilon _{j^{\prime }}=\sum_{t=1}^{d_{n}}\Lambda _{n\left( t\right)
}\log Z_{n\left( t\right) }+\Lambda _{n}\epsilon _{n,j^{\prime }}%
\end{array}%
\right. ,
\end{equation*}%
where $\left( \epsilon _{n,j},\epsilon _{n,j^{\prime }}\right) $ are
independent and are distributed according to a $\mathcal{G}\left( 0,1\right)
$ distribution. Therefore
\begin{equation*}
cov\left( \varepsilon _{j},\varepsilon _{j^{\prime }}\right) =var\left(
\sum_{t=1}^{d_{n}}\Lambda _{n\left( t\right) }\log Z_{n\left( t\right)
}\right) =\left( 1-\Lambda _{n}^{2}\right) var\left( \varepsilon \right) ,
\end{equation*}%
which established the formula.
\end{proof}

\section{Correlation structure of the multivariate Fr\'{e}chet distribution
with Gumbel copula\label{sec:frechet}}

We can use the previous results in order to provide an explicit formula for
the correlation between two Fr\'{e}chet random variables coupled with a
Gumbel copula of parameter $1/\lambda $, a formula which can be useful, see
Mattsson et al. (2014), Fosgerau et al. (2018), and H\aa rsman and Mattsson
(2020)\footnote{%
The author thanks L.-G. Mattsson for an insightful question leading to this
section.}. Recall that a Gumbel copula of parameter $1/\lambda $ is
expressed as $C\left( u,v\right) =\exp \left( -\left( \left( \left( -\log
u\right) ^{1/\lambda }+\left( -\log v\right) ^{1/\lambda }\right) \right)
^{\lambda }\right) $. Also recall that a Fr\'{e}chet random variable of
shape parameter $\alpha >0$ has c.d.f. $F\left( x\right) =\exp \left(
-x^{-\alpha }\right) $; it can be represented as $\exp \left( \varepsilon
/\alpha \right) $, where $\varepsilon \sim \mathcal{G}\left( 0,1\right) $.
It is easy to see that a pair $\left( \delta _{1},\delta _{2}\right) $ of Fr%
\'{e}chet random variables coupled with a Gumbel copula of parameter $%
1/\lambda $ can be written as $\left( \delta _{1}=\exp \left( \varepsilon
_{1}/\alpha \right) ,\delta _{2}=\exp \left( \varepsilon _{2}/\alpha \right)
\right) $, where the joint c.d.f. of $\left( \varepsilon _{1},\varepsilon
_{2}\right) $ is $\exp \left( -\left( \exp \left( -x/\lambda \right) +\exp
\left( -y/\lambda \right) \right) ^{\lambda }\right) $, falling under
specification~(\ref{nestedLogit}). By proposition~\ref{prop:representation},
it follows that we can represent $\left( \delta _{1},\delta _{2}\right) $ as%
\begin{equation*}
\delta _{1}=\exp \left( \frac{\lambda }{\alpha }\left( \epsilon _{1}+\log
Z\right) \right) ,\delta _{1}=\exp \left( \frac{\lambda }{\alpha }\left(
\epsilon _{2}+\log Z\right) \right)
\end{equation*}%
where $\epsilon _{1},\epsilon _{2}\sim \mathcal{G}\left( 0,1\right) $, $%
Z\sim \mathcal{P}\left( \lambda \right) $, and $\left( \epsilon
_{1},\epsilon _{2},Z\right) $ are independent. We can therefore provide a
close-form formula for the correlation between $\delta _{1}$ and $\delta
_{2} $.

\begin{proposition}
\label{prop:correlFrechet}Assume $\alpha >2$ and $\lambda \in \left(
0,1\right) $. Then the correlation between two Fr\'{e}chet random variables
of shape parameter $\alpha $ coupled by a Gumbel copula of parameter $%
1/\lambda $ is given by%
\begin{equation}
\rho \left( \delta _{1},\delta _{2}\right) =\frac{\Gamma \left( 1-2/\alpha
\right) \Gamma \left( 1-\frac{\lambda }{\alpha }\right) ^{2}/\Gamma \left(
1-2\lambda /\alpha \right) -\Gamma \left( 1-1/\alpha \right) ^{2}}{\Gamma
\left( 1-2/\alpha \right) -\Gamma \left( 1-1/\alpha \right) ^{2}}.
\label{correlFrechetFormula}
\end{equation}
\end{proposition}

\begin{proof}
We have $\rho \left( \delta _{1},\delta _{2}\right) =\left( \mathbb{E}\left[
\delta _{1}\delta _{2}\right] -\mathbb{E}\left[ \delta _{1}\right]
^{2}\right) /\left( \mathbb{E}\left[ \delta _{1}^{2}\right] -\mathbb{E}\left[
\delta _{1}\right] ^{2}\right) $. Using the well-known fact that the $%
\mathcal{G}\left( 0,1\right) $ distribution has moment generating function $%
\Gamma \left( 1-t\right) $, where $\Gamma $, we get that $\mathbb{E}\left[
\delta _{1}\right] =\Gamma \left( 1-1/\alpha \right) $, and that $\mathbb{E}%
\left[ \delta _{1}^{2}\right] =\Gamma \left( 1-2/\alpha \right) $ $\mathbb{E}%
\left[ \delta _{1}\delta _{2}\right] =\mathbb{E}\left[ \exp \left( \frac{%
\lambda }{\alpha }\left( \epsilon _{1}+\epsilon _{2}+2\log Z\right) \right) %
\right] =\Gamma \left( 1-\frac{\lambda }{\alpha }\right) ^{2}\mathbb{E}\left[
Z^{2\lambda /\alpha }\right] $. Further, one has by formula~(\ref{momentsOfZ}%
) that $\mathbb{E}\left[ Z^{2\lambda /\alpha }\right] =\Gamma \left(
1-2/\alpha \right) /\Gamma \left( 1-2\lambda /\alpha \right) $, from which
we deduce~(\ref{correlFrechetFormula}).
\end{proof}

\section{Conclusion}

This paper is the first to provide a complete and rigorous random utility
representation of the nested logit model, solving a four-decade old problem
since Ben-Akiva's seminal work on the topic. The representation will
hopefully shed new light on the nested logit model and help formulate
specifications as well as interpret estimated nested logit models. Since
this logit model is popular in practical work, the representation can be
useful in many applications, as in the example given in section~\ref%
{sec:frechet}. An additional contribution of the paper is the complete
derivation of the general recursive nested logit model based on trees,
carried out in section~\ref{par:multiple-nests}, which is typically not
performed in full in expository texts. We feel that the standard textbook
treatment of the nested logit model would benefit from adopting the present
approach.


\begin{thebibliography}{99}
\bibitem{BenAkiva73} Ben-Akiva, M. (1973). \emph{The structure of travel
demand models}. PhD thesis, MIT.

\bibitem{BenAkivaLerman85} Ben-Akiva, M., and Lerman, S. (1985). \emph{%
Discrete Choice Analysis: Theory and Application to Travel Demand}. MIT\
Press.

\bibitem{Bochner37} Bochner, S. (1937). \textquotedblleft Completely
monotone functions of the Laplace operator for torus and
sphere.\textquotedblright\ \emph{Duke Mathematical Journal} vol. 3, pp.
488-502.

\bibitem{Cardell} Cardell, S. (1997). \textquotedblleft Variance Components
Structures for the Extreme-Value and Logistic Distributions with Application
to Models of Heterogeneity.\textquotedblright\ \emph{Econometric Theory}
13(2), pp. 185--213.

\bibitem{DishonBendler90} Dishon, M. and Bendler, J. (1990).
\textquotedblleft Tables of the inverse Laplace transform of the function \$
e\symbol{94}\{-s\symbol{94}\TEXTsymbol{\backslash}beta\}
\$.\textquotedblright\ \emph{Journal of Research of the National Institute
of Standards and Technology} 95, pp. 433--467.

\bibitem{GalichonSalanie2015} Galichon, A., and Salani\'{e}, B. (2015).
\textquotedblleft Cupid's invisible hand.\textquotedblright\ Working paper.

\bibitem{Feller71} Feller, W. (1971). \emph{An Introduction to Probability
Theory and its Applications} vol. 2, 2nd edition. Wiley.

\bibitem{Fosgerau2018} Fosgerau M., Lindberg P.-O., Mattsson L.-G., Weibull,
J. (2018). \textquotedblleft A note on the invariance of the distribution of
the maximum.\textquotedblright\ \emph{Journal of Mathematical Economics} 74
pp. 56--61.

\bibitem{HarsmanMattsson2020} H\aa rsman, B. and Mattsson, L.-G. (2020).
\textquotedblleft Analyzing the returns to entrepreneurship by a modified
Lazear model.\textquotedblright\ Forthcoming, \emph{Small Business Economics}%
. https://doi.org/10.1007/s11187-020-00377-1.

\bibitem{Humbert45} Humbert, P. (1945). \textquotedblleft Nouvelles
correspondances symboliques\textquotedblright . \emph{Bulletin de la Soci%
\'{e}t\'{e} Math\'{e}matique de France} 69, pp. 121--129.

\bibitem{Mattsson2014} Mattsson, L.-G., Weibull, J. Lindberg, P.-O. (2014).
\textquotedblleft Extreme values, invariance and choice
probabilities.\textquotedblright\ \emph{Transportation Research Part B} 59,
pp. 81--95.

\bibitem{mcf1978} {\ McFadden, D. (1978). \textquotedblleft Modeling the
choice of residential location\textquotedblright . In A. Karlquist et. al.,
editor, \emph{Spatial Interaction Theory and Residential Location}. North
Holland.}

\bibitem{Pollard46} Pollard, H. (1946). \textquotedblleft The representation
of \$ e\symbol{94}\{-x\symbol{94}\TEXTsymbol{\backslash}lambda\} \$ as a
Laplace integral\textquotedblright . \emph{Bulletin of the American
Mathematical Society} 52(10), pp. 908--910.

\bibitem{Ridout09} Ridout, M. S. (2009). \textquotedblleft Generating random
numbers from a distribution specified by its Laplace
transform\textquotedblright . Statistics and Computing 19, pp 439--450.

\bibitem{Tiago58} Tiago de Oliveira, J. (1958). \textquotedblleft Extremal
Distributions\textquotedblright . \emph{Revista de Faculdada du Ciencia,
Lisboa, Serie A}, Vol. 7, pp. 215--227.

\bibitem{Tiago97} Tiago de Oliveira, J. (1997). \emph{Statistical analysis
of the extreme}. Pendor.
\end{thebibliography}
\end{document}